\documentclass{amsart}
\usepackage[utf8]{inputenc}
\usepackage{amssymb}
\usepackage{amsthm}
\usepackage{tikz}
\usepackage{amsmath}
\usepackage{subfigure}
\usepackage{theoremref}
\usepackage{mathtools}
\usepackage[round]{natbib}   
\usepackage[section]{placeins}

\tikzstyle{vertex}=[circle, draw, inner sep=0pt, minimum size=6pt]
\newtheorem{theorem}{Theorem}[section] 
 
\newtheorem{lemma}[theorem]{Lemma}

\theoremstyle{definition}
\newtheorem{definition}{Definition}[section]
\newtheorem{problem}[theorem]{Problem}

\theoremstyle{remark}
\newtheorem{remark}[theorem]{Remark}

\title[Partial Partition Complex]{The Partial Partition Complex}

\author{Michael J. Gottstein}

\address{Department of Mathematics and Computer Science,
  Marywood University, Scranton, Pennsylvania, USA}
\keywords{partial partitions, simplicial complexes, shellability,
  vertex decomposability, nonpure complexes, cross-polytopes,
  automorphism groups}

\begin{document}

\begin{abstract}
The set of partial partitions of $\{1,\ldots,n\}$, ordered by containment,
forms an abstract simplicial complex $D_n$ whose vertices are the nonempty
subsets of $\{1,\ldots,n\}$ and whose simplices are collections of pairwise
disjoint subsets. We prove that $D_n$ is vertex-decomposable, give an
explicit nonpure shelling, and use it to compute the reduced
homology: for $1 \le j \le n$, the homology in dimension $j-1$ is free
abelian of rank equal to the number of partitions of $\{1,\ldots,n\}$ into
$j$ blocks containing no singleton blocks. Explicit generators are
constructed as boundary complexes of $j$-dimensional cross-polytopes, one
for each non-singleton partition. We also prove that the automorphism group
of $D_n$ is the symmetric group on $n$ letters.
\end{abstract}

\maketitle

\section{Introduction}
Let $[n]$ be the set of positive integers $1,\ldots, n$. A partial partition of $[n]$ is a partition of a subset $I$ of $[n]$ (including the empty set). We denote the set of all partial partitions of $[n]$ by $D_n$. A partial partition will be denoted by dropping the braces around each block and dropping the commas within each block. For example, $\{\{1,2\},\{3\}\}$ is written as $\{12,3\}$. A \emph{singleton block} is a block of a partition with size equal to one. A \emph{singleton partition} is a partition that contains at least one singleton block; conversely, a \emph{non-singleton partition} is one in which every block has size at least two. When ordered by set inclusion, $D_n$ is an abstract simplicial complex with facets the partitions of $[n]$ and ground set the power set of $[n]$. This is different than the typical orderings on the partial partitions of $[n]$, such as the refinement order that appears in the lattice of set partitions \citep{StanleyEC1}.

The study of $D_n$ was inspired by the Rhodes semilattice, which arises in the complexity theory of semigroups and automata \citep{Rhodes}. The Rhodes semilattice plays an important role in decompositions associated with finite semigroups and automata; see also the Krohn--Rhodes theory background in \citet{RhodesSteinberg}. In the case of the trivial group, the Rhodes semilattice simplifies substantially, and the underlying set is precisely the set of partial partitions considered here. In that setting, the Rhodes semilattice is ordered by refinement, whereas in this paper, we study the same set ordered by containment. Thus, $D_n$ can be viewed as a simplicial-complex simplification of the trivial-group Rhodes semilattice.

We determine the homology of $D_n$ and provide an explicit formula for its Betti numbers, which count partitions without singleton blocks. More precisely, if $D(n,j,j)$ denotes the number of partitions of $[n]$ into $j$ blocks with no singleton blocks, then
\[
\widetilde{H}_{j-1}(D_n)\cong \mathbb{Z}^{D(n,j,j)}.
\]
We prove this by showing that $D_n$ is a shellable nonpure complex in the sense of \citet{wachs}. In the shelling order used here, the facets that contribute to homology are exactly the non-singleton partitions.

We also prove that there is a basis of $\widetilde{H}_{j-1}(D_n)$ whose elements can be represented as the boundary complexes of $j$-dimensional cross-polytopes \citep{Ziegler}. These cycles are constructed explicitly from non-singleton partitions. The resulting homology generators thus admit a concrete combinatorial description. The enumeration of partitions without singleton blocks is related to the work of \citet{bona}, who obtain these numbers by the use of generating functions.

In addition to the shelling argument, we record a recursive structural property of $D_n$. If
\[
\sigma=\{B_1,\ldots,B_k\}
\]
is a simplex of $D_n$, then
\[
\operatorname{link}_{D_n}(\sigma)\cong D_{n-\left|\bigcup_{i=1}^k B_i\right|}.
\]
Thus, every link in $D_n$ is again a partial partition complex on the unused elements. This link recursion provides a direct proof that $D_n$ is vertex-decomposable \citep{ProvanBillera,Wood}.

We also show that the automorphism group of $D_n$ is isomorphic to $S_n$. Finally, we formulate natural geometric realization questions for $D_n$.

\section{Examples}

The figures below give geometric realizations of $D_2$, $D_3$, and $D_4$.
These examples provide useful intuition for the homology of $D_n$ in small
cases.

When $n=0$, we have
\[
D_0=\{\emptyset\}.
\]
When $n=1$, we have
\[
D_1=\{\emptyset,\{1\}\},
\]
and hence $\widetilde{H}_{0}(D_1)$ is trivial. Geometric realizations of
$D_2$ and $D_3$ are shown in Figure~\ref{F1}. In these figures, every face is
labeled by its corresponding partial partition.

A realization of $D_4$ is shown in Figure~\ref{F2}. In this case, to make the
figure more readable, we label only the vertices. Each of these realizations
has two connected components. One component is a single point corresponding to
the \emph{trivial partition}, namely the partition with one block. The other component
contains the \emph{total partition}, namely the partition with $n$ singleton blocks.
This phenomenon holds for $D_n$ in general. Therefore,
\[
\widetilde{H}_{0}(D_n)\cong \mathbb{Z}
\qquad \text{for } n\geq 2.
\]
This is consistent with Theorem~\ref{T:homology-main} below: the only partition
of $[n]$ into a single non-singleton block is $\{[n]\}$ itself, so
$D(n,1,1)=1$ for $n\geq 2$.

For $n=2$ and $n=3$, the component containing the total partition is
contractible. Thus
\[
\widetilde{H}_{1}(D_2), \qquad
\widetilde{H}_{1}(D_3), \qquad
\widetilde{H}_{2}(D_3)
\]
are all trivial. The situation changes for $D_4$. In this case, the component
containing the total partition is not contractible. It contains exactly three
independent $1$-dimensional holes. Each of these holes contains a unique
non-singleton partition of size $2$, which is colored orange in Figure~\ref{F2}.

This suggests a natural basis for $\widetilde{H}_{1}(D_4)$: one basis element
for each non-singleton partition of size $2$. More precisely, one may ask
whether, for each non-singleton partition of size $2$ in $D_4$, there exists a
$1$-cycle containing exactly one such partition, and whether the homology
classes of these cycles form a basis for $\widetilde{H}_{1}(D_4)$. The answer
is yes. Consequently,
\[
\widetilde{H}_{1}(D_4)\cong \mathbb{Z}^3.
\]
More generally, the number of non-singleton partitions of size $2$ in $D_n$
is the Betti number of $\widetilde{H}_{1}(D_n)$.

We now reinterpret the basis elements for $\widetilde{H}_{1}(D_4)$ in a way
that will generalize to higher-dimensional homology. Each basis element
described above may be represented by the boundary of a $2$-dimensional
cross-polytope. Figure~\ref{F3}\subref{F3a} shows the subcomplex of $D_4$
containing the representatives associated to the non-singleton partition
$\{12,34\}$. Figure~\ref{F3}\subref{F3b} shows one such representative, and
Figure~\ref{F3}\subref{F3c} illustrates this cycle as the boundary of a
$2$-dimensional cross-polytope.

In this example, choosing such a representative amounts to choosing one
element from each block of the non-singleton partition $\{12,34\}$ and
forming the corresponding singleton blocks.
Equivalently, it is a choice of two singleton blocks from the total partition
$\{1,2,3,4\}$, one lying inside each block of $\{12,34\}$.

The same idea extends to higher dimensions. If we try to understand
$\widetilde{H}_{2}(D_n)$ using the same geometric intuition, then the
non-singleton partitions of size $3$ play the same role that the non-singleton
partitions of size $2$ played in the computation of $\widetilde{H}_{1}(D_n)$.
For example, in $D_6$, the partition
\[
\{12,34,56\}
\]
is a non-singleton partition of size $3$. By choosing one element from
each of the blocks $12$, $34$, and $56$, we obtain a cycle that represents a
nontrivial class in $\widetilde{H}_{2}(D_6)$. This cycle is the boundary of a
$3$-dimensional cross-polytope.

Figure~\ref{F4}\subref{F4a} illustrates such a choice of singleton blocks from
the total partition $\{1,2,3,4,5,6\}$. Figure~\ref{F4}\subref{F4b} shows how
this choice produces a nontrivial $2$-cycle in $D_6$ containing the partition
$\{12,34,56\}$. The precise general construction of these cycles will be given
later.

The preceding examples motivate the general homology calculation, which we prove using the theory of shellable nonpure complexes.

\begin{figure}[ht]
\centering
\subfigure[$D_2$\label{F1a}]{%
\begin{tikzpicture}[scale=3]
\node at (-.85,0){2};
\node at (-.85,1){1};
\node at (-1.15,.5){1,2};
\node at (-1.8,.5){12};
\node[circle, fill=black, inner sep=1.2pt] (1) at (-1,0){};
\node[circle, fill=black, inner sep=1.2pt] (2) at (-1,1){};
\node[circle, fill=black, inner sep=1.2pt] (3) at (-2,.5){};
\draw[blue!60] (1) -- (2);
\end{tikzpicture}%
}
\subfigure[$D_3$\label{F1b}]{%
\resizebox{0.4\textwidth}{!}{%
\begin{tikzpicture}[scale=2/1.414, join=round]
\tikzstyle{conefill} = [fill=blue!40,fill opacity=0.5]
\node at (-2,.65){123};
\node at (.15,1.5){1};
\node at (-1.016,.05){2};
\node at (1.016,.05){3};
\node at (0,2.65){23};
\node at (-1.403,-.75){13};
\node at (1.403,-.75){12};
\node at (1.50,-.3){3,12};
\node at (-1.50,-.3){2,13};
\node at (-.3,2){1,23};
\node at (0,.65){1,2,3};
\node at (0,-.2){2,3};
\node at (.7,.7){1,3};
\node at (-.7,.7){1,2};
\node[circle, fill=black, inner sep=1.2pt] (1) at (-2.4,.63){};
\node[circle, fill=black, inner sep=1.2pt] (1) at (0,1.5){};
\node[circle, fill=black, inner sep=1.2pt](2) at (-0.866,0){};
\node[circle, fill=black, inner sep=1.2pt](3) at (0.866,0){};
\node[circle, fill=black, inner sep=1.2pt] (4) at (0,2.5){};
\node[circle, fill=black, inner sep=1.2pt] (5) at (-1.573,-.707){};
\node[circle, fill=black, inner sep=1.2pt] (6) at (1.573,-.707){};
\filldraw[conefill] (0,1.5)--(-0.866,0)--(0.866,0)--cycle;
\draw[green!80] (1) -- (4);
\draw[green!80] (2) -- (5);
\draw[green!80] (3) -- (6);
\end{tikzpicture}%
}%
}
\caption{Every face is labeled with its corresponding partial partition. We drop the outermost braces for clarity. For example, $\{12\}$ is written as $12$ and $\{12,3\}$ is written as $12,3$.}
\label{F1}
\end{figure}
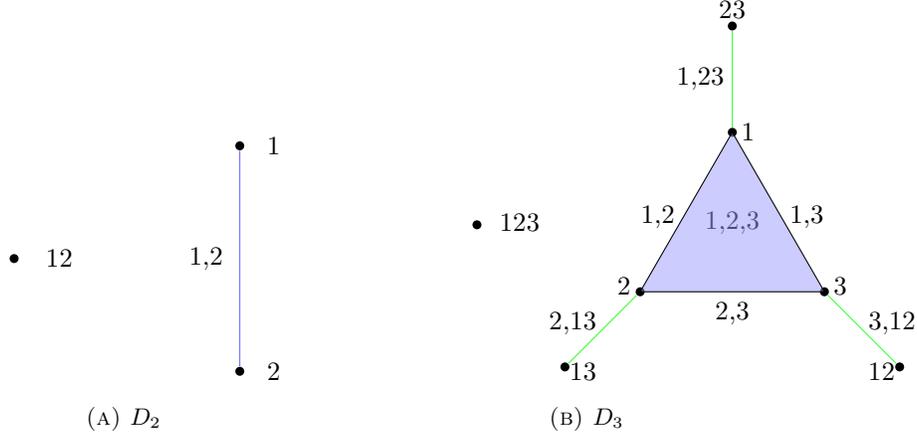

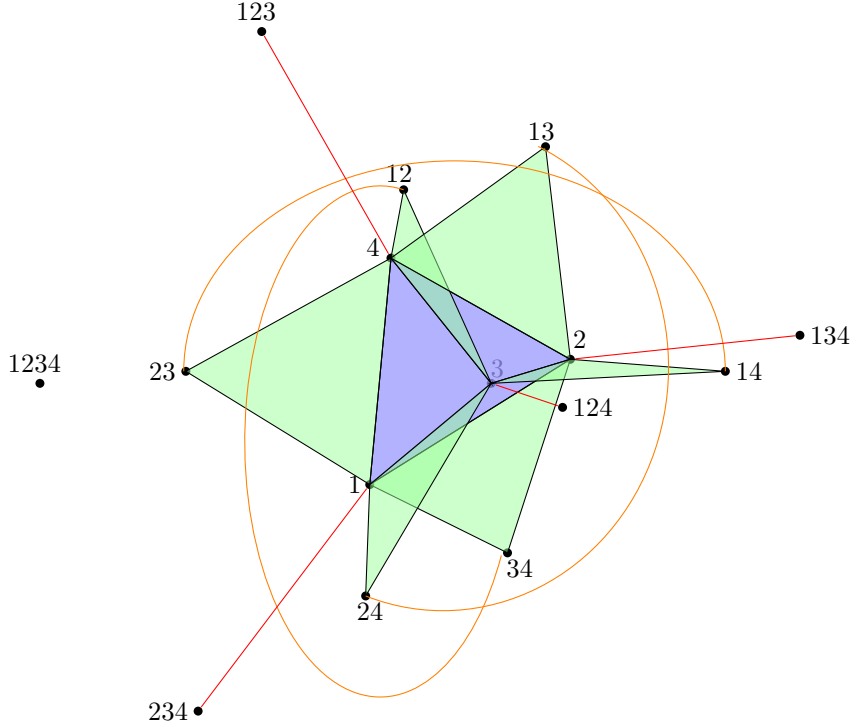
\begin{figure}[ht]
\centering
\begin{tikzpicture}[scale=1.7,  rotate around x=30]
\tikzstyle{conefill} = [fill=blue!40,fill opacity=0.5]
\tikzstyle{conefill1} = [fill=green!40,fill opacity=0.5]

\node[circle, fill=black, inner sep=1.2pt] (0) at (-2,1,0){};
\node[circle, inner sep=1.2pt] at (-2,1.2,0){1234};
\node[circle, fill=black, inner sep=1.2pt] (1) at (0,0,0){};
\node[circle,  inner sep=1.2pt] at (-0.1,0,0){1};
\node[circle, fill=black, inner sep=1.2pt] (2) at (1,0,-1){};
\node[circle,  inner sep=1.2pt] at (1.1,0.2,-1){2};
\node[circle, fill=black, inner sep=1.2pt] (3) at (1,1,0){};
\node[circle,  inner sep=1.2pt] at (1.07,1.15,0){3};
\node[circle, fill=black, inner sep=1.2pt] (4) at (0,1,-1){};
\node[circle,  inner sep=1.2pt] at (-.1,1.1,-1){4};
\node[circle, fill=black, inner sep=1.2pt] (5) at (.5,-1.293,-.5){};
\node[circle,  inner sep=1.2pt] at (.55,-1.45,-.5){34};
\node[circle, fill=black, inner sep=1.2pt] (6) at (.5,2.293,-.5){};
\node[circle,  inner sep=1.2pt] at (.5,2.45,-.5){12};
\node[circle, fill=black, inner sep=1.2pt] (7) at (-1.293,.5,-.5){};
\node[circle,  inner sep=1.2pt] at (-1.45,.5,-.5){23};
\node[circle, fill=black, inner sep=1.2pt] (8) at ( 2.293,.5,-.5){};
\node[circle,  inner sep=1.2pt] at ( 2.45,.5,-.5){14};
\node[circle, fill=black, inner sep=1.2pt] (9) at (.5, .5, 1.293){};
\node[circle,  inner sep=1.2pt] at (.5, .35, 1.293){24};
\node[circle, fill=black, inner sep=1.2pt] (10) at (.5, .5,-2.293){};
\node[circle,  inner sep=1.2pt] at (.5, .65,-2.293){13};
\node[circle, fill=black, inner sep=1.2pt] (11) at (-1,-1,1){};
\node[circle,  inner sep=1.2pt] at (-1.2,-1,1){234};
\node[circle, fill=black, inner sep=1.2pt] (12) at (2,-1,-2){};
\node[circle,  inner sep=1.2pt] at (2.2,-1,-2){134};
\node[circle, fill=black, inner sep=1.2pt] (13) at (2,2,1){};
\node[circle,  inner sep=1.2pt]  at (2.2,2,1){124};
\node[circle, fill=black, inner sep=1.2pt] (14) at (-1,2,-2){};
\node[circle,  inner sep=1.2pt]  at (-1,2.2,-2){123};

\filldraw[conefill] (0,0,0) -- (1,0,-1) -- (1,1,0) -- cycle ;
\filldraw[conefill] (0,0,0) -- (0,1,-1) -- (1,1,0)-- cycle ;
\filldraw[conefill] (0,0,0) -- (1,0,-1) -- (0,1,-1)-- cycle;
\filldraw[conefill] (1,0,-1) -- (1,1,0) -- (0,1,-1)-- cycle;
\filldraw[conefill1] (0,0,0) -- (.5,-1.293,-.5) -- (1,0,-1) --cycle;
\filldraw[conefill1] (1,1,0) -- (.5,2.293,-.5) -- (0,1,-1) -- cycle;
\filldraw[conefill1] (0,0,0) -- (-1.293,.5,-.5) -- (0,1,-1) -- cycle;
\filldraw[conefill1] (1,0,-1) -- ( 2.293,.5,-.5) -- (1,1,0) -- cycle;
\filldraw[conefill1] (0,0,0) -- (.5, .5, 1.293) -- (1,1,0) -- cycle;
\filldraw[conefill1] (0,1,-1) -- (.5, .5,-2.293) -- (1,0,-1) -- cycle;

\draw[orange] ( 2.293,.5,-.5) arc (0:180:1.8cm and 1.4cm);
\draw[orange] (.5,.5, 1.293) arc (-110:65:1.5cm and 1.62cm);
\draw[orange] ( .5, 2.293, -.5) arc (-280:-26.5:.9cm and 1.7cm);
\draw[red](1) -- (11);
\draw[red](2) -- (12);
\draw[red](3) -- (13);
\draw[red] (4) -- (14);
\end{tikzpicture}
\caption{$D_4$. The blue solid tetrahedron corresponds to the total partition $\{1, 2, 3, 4\}$. The green solid $2$-faces correspond to the singleton partitions $\{1, 2, 34\}$, $\{1, 3, 24\}$, $\{1, 4,  23\}$, $\{2, 3, 14\}$, $\{2, 4, 13\}$, $\{3, 4, 12\}$. The curved orange $1$-faces correspond to the non-singleton partitions $\{12, 34\}$, $\{13, 24\}$, $\{14, 23\}$. The remaining $1$-faces colored in red correspond to the singleton partitions $\{1, 234\}$, $\{2, 134\}$, $\{3, 124\}$, $\{4, 123\}$. The dot corresponds to the trivial partition $\{1234\}$.}
\label{F2}
\end{figure}

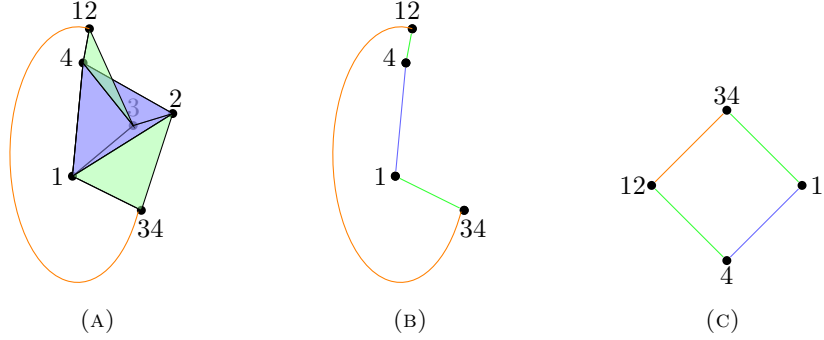
\begin{figure}[ht]
\centering
\subfigure[\label{F3a}]{%
\begin{tikzpicture}[scale=1,  rotate around x=30]
\tikzstyle{conefill} = [fill=blue!40,fill opacity=0.5]
\tikzstyle{conefill1} = [fill=green!40,fill opacity=0.5]

\node[circle, fill=black, inner sep=1.2pt] (1) at (0,0,0){};
\node[circle,  inner sep=1.2pt] at (-0.2,0,0){1};
\node[circle, fill=black, inner sep=1.2pt] (2) at (1,0,-1){};
\node[circle,  inner sep=1.2pt] at (1.1,0.3,-1){2};
\node[circle, fill=black, inner sep=1.2pt] (3) at (1,1,0){};
\node[circle,  inner sep=1.2pt] at (1.07,1.35,0){3};
\node[circle, fill=black, inner sep=1.2pt] (4) at (0,1,-1){};
\node[circle,  inner sep=1.2pt] at (-.2,1.1,-1){4};
\node[circle, fill=black, inner sep=1.2pt] (5) at (.5,-1.293,-.5){};
\node[circle,  inner sep=1.2pt] at (.55,-1.65,-.5){34};
\node[circle, fill=black, inner sep=1.2pt] (6) at (.5,2.293,-.5){};
\node[circle,  inner sep=1.2pt] at (.5,2.65,-.5){12};

\filldraw[conefill] (0,0,0) -- (1,0,-1) -- (1,1,0) -- cycle ;
\filldraw[conefill] (0,0,0) -- (0,1,-1) -- (1,1,0)-- cycle ;
\filldraw[conefill] (0,0,0) -- (1,0,-1) -- (0,1,-1)-- cycle;
\filldraw[conefill] (1,0,-1) -- (1,1,0) -- (0,1,-1)-- cycle;
\filldraw[conefill1] (0,0,0) -- (.5,-1.293,-.5) -- (1,0,-1) --cycle;
\filldraw[conefill1] (1,1,0) -- (.5,2.293,-.5) -- (0,1,-1) -- cycle;

\draw[black] (5) -- (1);
\draw[black] (1) -- (4);
\draw[black] (4) -- (6);
\draw[orange] ( .5, 2.293, -.5) arc (-280:-26.5:.9cm and 1.7cm);
\end{tikzpicture}%
}\hfil
\subfigure[\label{F3b}]{%
\begin{tikzpicture}[scale=1,  rotate around x=30]
\tikzstyle{conefill} = [fill=blue!40,fill opacity=0.5]
\tikzstyle{conefill1} = [fill=green!20,fill opacity=0.5]

\node[circle, fill=black, inner sep=1.2pt] (1) at (0,0,0){};
\node[circle, fill=black, inner sep=1.2pt] (4) at (0,1,-1){};
\node[circle, fill=black, inner sep=1.2pt] (5) at (.5,-1.293,-.5){};
\node[circle, fill=black, inner sep=1.2pt] (6) at (.5,2.293,-.5){};

\node[circle, fill=black, inner sep=1.2pt] (1) at (0,0,0){};
\node[circle,  inner sep=1.2pt] at (-0.2,0,0){1};
\node[circle, fill=black, inner sep=1.2pt] (4) at (0,1,-1){};
\node[circle,  inner sep=1.2pt] at (-.2,1.1,-1){4};
\node[circle, fill=black, inner sep=1.2pt] (5) at (.5,-1.293,-.5){};
\node[circle,  inner sep=1.2pt] at (.55,-1.65,-.5){34};
\node[circle, fill=black, inner sep=1.2pt] (6) at (.5,2.293,-.5){};
\node[circle,  inner sep=1.2pt] at (.5,2.65,-.5){12};

\draw[green!80] (5) -- (1);
\draw[blue!60] (1) -- (4);
\draw[green!80] (4) -- (6);

\draw[orange] ( .5, 2.293, -.5) arc (-280:-26.5:.9cm and 1.7cm);
\end{tikzpicture}%
}\hfil
\subfigure[\label{F3c}]{%
\begin{tikzpicture}[scale=1]
\node[circle, fill=black, inner sep=1.2pt] (3) at (1,0){};
\node[circle, fill=black, inner sep=1.2pt] (4) at (-1,0){};
\node[circle, fill=black, inner sep=1.2pt] (5) at (0,1){};
\node[circle, fill=black, inner sep=1.2pt] (6) at (0,-1){};

\node[circle, inner sep=1.2pt] at (1.2,0){1};
\node[circle, inner sep=1.2pt] at (-1.25,0){12};
\node[circle, inner sep=1.2pt] at (0,1.2){34};
\node[circle, inner sep=1.2pt] at (0,-1.2){4};

\draw[orange] (5) -- (4);
\draw[green!80] (5) -- (3);
\draw[blue!60] (3) -- (6);
\draw[green!80] (4) -- (6);
\end{tikzpicture}%
}
\caption{The orange $1$-face is the size $2$ non-singleton partition $\{12,34\}$. The blue $1$-face in (b) is the partial partition $\{1,4\}$. The two green $1$-faces are the partial partitions $\{1,34\}$ and $\{12,4\}$.}
\label{F3}
\end{figure}

\begin{figure}[htbp]
\centering
\subfigure[\label{F4a}]{%
\resizebox{0.3\textwidth}{!}{%
\begin{tikzpicture}[scale=.6, rotate around x=40, rotate around y=45, rotate around z=0]
\tikzstyle{conefill} = [fill=orange!80,fill opacity=0.5]
\tikzstyle{conefill1} = [fill=blue!40,fill opacity=0.5]

\node at (-1.9,-.8,-3){12};
\node at (1.7,-.6,-3){34};
\node at (-0.2,-3.9,-3){56};
\node[circle, fill=black, inner sep=1.2pt] (12) at (-1.5,-1,-3){};
\node[circle, fill=black, inner sep=1.2pt] (34) at (1.5,-1,-3){};
\node[circle, fill=black, inner sep=1.2pt] (56) at (0,-3.5,-3){};

\draw[black] (12) -- (34);
\draw[black] (12) -- (56);
\draw[black] (34) -- (56);

\filldraw[conefill] (-1.5,-1,-3) -- (1.5,-1,-3) -- (0,-3.5,-3) -- cycle;

\node at (0.1,.6,0){5};
\node at (-1.5,-3.1,0){3};
\node at (1.8,-2.7,0){1};
\node at (1.7,-.4,0){6};
\node at (0,-4.2,0){2};
\node at (-1.7,-.8,0){4};

\node[circle, fill=black, inner sep=1.2pt] (p0) at (0,0,0){};
\node[circle, fill=black, inner sep=1.2pt] (p1) at (-1.5,-1,0){};
\node[circle, fill=black, inner sep=1.2pt] (p2) at (-1.5,-2.5,0){};
\node[circle, fill=black, inner sep=1.2pt] (p3) at (1.5,-2.5,0){};
\node[circle, fill=black, inner sep=1.2pt] (p4) at (1.5,-1,0){};
\node[circle, fill=black, inner sep=1.2pt] (p5) at (0,-3.5,0){};

\draw[black] (p0) -- (p1);
\draw[black] (p0) -- (p2);
\draw[black] (p0) -- (p3);
\draw[black] (p0) -- (p4);
\draw[black] (p0) -- (p5);
\draw[black] (p1) -- (p2);
\draw[black] (p1) -- (p3);
\draw[black] (p1) -- (p4);
\draw[black] (p1) -- (p5);
\draw[black] (p2) -- (p3);
\draw[black] (p2) -- (p4);
\draw[black] (p2) -- (p5);
\draw[black] (p3) -- (p4);
\draw[black] (p3) -- (p5);
\draw[black] (p4) -- (p5);

\filldraw[conefill1] (0,0) -- (-1.5,-2.5) -- (1.5,-2.5) -- cycle;
\end{tikzpicture}%
}%
}\hfil
\subfigure[\label{F4b}]{%
\resizebox{0.3\textwidth}{!}{%
\begin{tikzpicture}[scale=1.7,line join=bevel,z=-5.5]

\node[circle, fill=black, inner sep=1.2pt] (A1) at (0,0,-1){};
\node[circle, fill=black, inner sep=1.2pt] (A2) at (-1,0,0){};
\node[circle, fill=black, inner sep=1.2pt] (A3) at (0,0,1){};
\node[circle, fill=black, inner sep=1.2pt] (A4) at (1,0,0){};
\node[circle, fill=black, inner sep=1.2pt] (B1) at (0,1,0){};
\node[circle, fill=black, inner sep=1.2pt] (C1) at (0,-1,0){};

\node[circle, inner sep=1.2pt] at (0.15,.15,-1){5};
\node[circle, inner sep=1.2pt] at (-1.2,0,0){12};
\node[circle, inner sep=1.2pt] at (-0.15,-0.15,1){56};
\node[circle, inner sep=1.2pt] at (1.2,0,0){1};
\node[circle, inner sep=1.2pt] at (0,1.2,0){34};
\node[circle, inner sep=1.2pt] at (0,-1.2,0){3};

\coordinate (A1) at (0,0,-1);
\coordinate (A2) at (-1,0,0);
\coordinate (A3) at (0,0,1);
\coordinate (A4) at (1,0,0);
\coordinate (B1) at (0,1,0);
\coordinate (C1) at (0,-1,0);

\draw (A1) -- (A2) -- (B1) -- cycle;
\draw (A4) -- (A1) -- (B1) -- cycle;
\draw (A1) -- (A2) -- (C1) -- cycle;
\draw [fill opacity=0.5,fill=blue!40] (A4) -- (A1) -- (C1) -- cycle;
\draw [fill opacity=0.5,fill=orange!80] (A2) -- (A3) -- (B1) -- cycle;
\draw [fill opacity=0.5,fill=green!40] (A3) -- (A4) -- (B1) -- cycle;
\draw [fill opacity=0.5,fill=green!40] (A2) -- (A3) -- (C1) -- cycle;
\draw [fill opacity=0.5,fill=green!40] (A3) -- (A4) -- (C1) -- cycle;
\end{tikzpicture}%
}%
}
\caption{The orange triangle is the size $3$ non-singleton partition $\{12,34,56\}$. The blue triangle is the partial partition $\{1,3,5\}$. The green triangles are the partial partitions formed by interchanging the non-singleton blocks with the corresponding singleton blocks in all possible ways.}
\label{F4}
\end{figure}
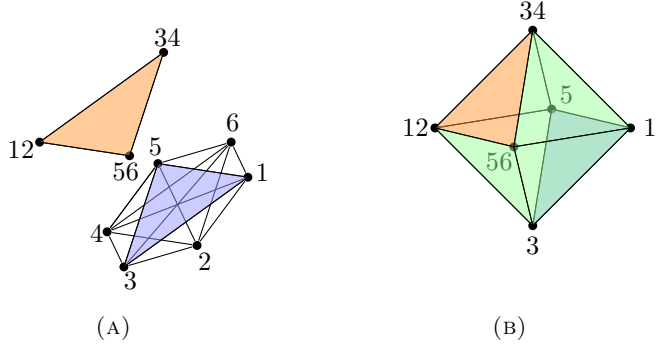

\FloatBarrier

\section{Link Recursion and Vertex Decomposability}

We now record one of the main structural properties of the complex $D_n$: the link of any simplex in $D_n$ is again a partial partition complex on a smaller ground set. We use the standard definition of the link of a face in an abstract simplicial complex \citep{Munkres}.


\begin{theorem}[Link recursion]\label{T:link-recursion}
Let $\sigma=\{B_1,\dots,B_k\} \in D_n$, and $U(\sigma)=\bigcup_{i=1}^k B_i$.
Then
\[
\operatorname{link}_{D_n}(\sigma)
\cong
D_{m},
\]
where $m=n-|U(\sigma)|$.
\end{theorem}

\begin{proof}
By definition,
\[
\operatorname{link}_{D_n}(\sigma)
=
\left\{
\tau\in D_n:
\tau\cap \sigma=\emptyset
\text{ and }
\tau\cup \sigma\in D_n
\right\}.
\]
Since simplices of $D_n$ are collections of pairwise disjoint nonempty subsets of $[n]$, the condition $\tau\cup\sigma\in D_n$
means that every block appearing in $\tau$ must be disjoint from every block appearing in $\sigma$.

Write $\sigma=\{B_1,\dots,B_k\}$
and let $U(\sigma)=\bigcup_{i=1}^k B_i$.
If $A$ is a vertex of $D_n$, then $A$ belongs to the link of $\sigma$ if and only if $\sigma\cup\{A\}$ is a simplex of $D_n$, which is true if and only if $A\cap B_i=\emptyset$ for every $i=1,\dots,k$. Equivalently, $A\subseteq [n]\setminus U(\sigma).$ Thus the vertices of $\operatorname{link}_{D_n}(\sigma)$ are precisely the nonempty subsets of the complement $[n]\setminus U(\sigma).$

Now consider a collection of vertices $A_1,\dots, A_r$ in the link of $\sigma$. These vertices form a simplex in $
\operatorname{link}_{D_n}(\sigma)$
if and only if $ \sigma\cup\{A_1,\dots,A_r\}$
is a simplex of $D_n$. Since each $A_j$ is already disjoint from every $B_i$, this is equivalent to requiring that $A_1,\dots,A_r$
are pairwise disjoint.

Therefore $\operatorname{link}_{D_n}(\sigma)$ is exactly the simplicial complex whose vertices are the nonempty subsets of
$[n]\setminus U(\sigma),$ and whose simplices are collections of pairwise disjoint such subsets. This is just the partial partition complex on the ground set $[n]\setminus U(\sigma).$
Hence,
\[
\operatorname{link}_{D_n}(\sigma)
\cong
D_{[n]\setminus U(\sigma)}.
\]

Finally, since $|[n]\setminus U(\sigma)|=n-|U(\sigma)|$, the complex $D_{[n]\setminus U(\sigma)}$ is isomorphic to $D_m$, where $m=n-|U(\sigma)|.$
\end{proof}

We now use this recursive structure to prove that $D_n$ is vertex decomposable.

Recall that a simplicial complex $\Delta$ is \emph{vertex decomposable} in the sense of \citet{ProvanBillera} if either $\Delta$ is a simplex, or there exists a vertex $v$ of $\Delta$ such that
\begin{enumerate}
    \item both $\Delta\setminus v$ and $\operatorname{link}_{\Delta}(v)$ are vertex decomposable, and
    \item no face of $\operatorname{link}_{\Delta}(v)$ is a facet of $\Delta\setminus v$.
\end{enumerate}
A vertex satisfying Condition (2) is called a \emph{shedding vertex}.

\begin{theorem}\label{T:vertex-decomposable}
The complex $D_n$ is vertex decomposable.
\end{theorem}

\begin{proof}
We argue by induction on $n$. The base cases $n=0$ and $n=1$ are immediate, since $D_0=\{\emptyset\}$ and $D_1=\{\emptyset,\{1\}\}$ are both simplices.

For $n\geq 2$, we give an explicit recursive deletion order. Order the non-singleton vertices of $D_n$, that is, the subsets
\[
B\subseteq [n],
\qquad
|B|\geq 2,
\]
by non-increasing cardinality, i.e., the vertices of larger cardinality appear first in the order, and vertices of the same size are ordered arbitrarily. We claim that deleting the non-singleton vertices in this order gives a vertex decomposition of $D_n$.

Let $B$ be a non-singleton vertex of largest cardinality at some stage of the deletion process, and let $\Delta$ denote the complex remaining at that stage.

We first check the shedding condition. Let $\tau\in \operatorname{link}_{\Delta}(B).$
Then
$
\tau\cup\{B\}
$
is a simplex of $\Delta$, and hence also a simplex of $D_n$. Therefore every vertex of $\tau$ is disjoint from $B$.

Since $B$ is non-singleton, choose an element
$
b\in B.
$
Then the singleton vertex
$
\{b\}
$
is disjoint from every vertex of $\tau$, because every vertex of $\tau$ is disjoint from all of $B$. Hence
$
\tau\cup\{\{b\}\}
$
is a simplex of $D_n$.

Moreover, the singleton vertex $\{b\}$ has not been deleted before this stage, because our deletion order deletes only non-singleton vertices. Also,
$
\{b\}\neq B
$
because $B$ is non-singleton. Therefore $\{b\}$ is a vertex of
$
\Delta\setminus B.
$
It follows that
$
\tau\cup\{\{b\}\}
$
is a face of $\Delta\setminus B$ properly containing $\tau$. Thus $\tau$ is not a facet of $\Delta\setminus B$.

Since $\tau$ was an arbitrary face of $\operatorname{link}_{\Delta}(B)$, no face of the link is a facet of the deletion. Therefore, $B$ is a shedding vertex.

It remains to verify the recursive condition. By Theorem~\ref{T:link-recursion},
$
\operatorname{link}_{D_n}(B)\cong D_{n-|B|}.
$
The current $\Delta$ is obtained from $D_n$ by deleting some non-singleton vertices of cardinality strictly greater than $|B|$, namely those that come earlier in the deletion order. Restricting the deletion process to vertices that are subsets of $[n]\setminus B$ (the vertices of the link), the previously deleted vertices form an initial segment of the analogous deletion order on $D_{n-|B|}$. Since $|B|\geq 2$, we have $n-|B|<n$, so by the inductive hypothesis $D_{n-|B|}$ is vertex decomposable, and the same recursive argument applies to $\operatorname{link}_\Delta(B)$ as a partial run of the deletion order on $D_{n-|B|}$.

Likewise, $\Delta\setminus B$ is obtained by continuing the same deletion process on $D_n$ after removing $B$, and the same induction applies.

When the process is complete, every non-singleton vertex has been deleted. The remaining complex has vertex set $\{\{1\},\dots,\{n\}\}$, and these singletons are pairwise disjoint, so they form an $(n-1)$-simplex. Hence $D_n$ is vertex decomposable.
\end{proof}

Although vertex decomposability implies shellability, we next give an explicit shelling order and prove shellability directly with respect to this order. This explicit shelling is useful because it allows us to enumerate the resulting homology basis elements.

\section{Shellability of $D_n$}

Suppose $\Delta$ is an abstract simplicial complex. For a face $F \in \Delta$ the \emph{closure} of $F$ is the smallest subcomplex of $\Delta$ containing $F$, equal to $\bar{F}=\{\tau \in \Delta \mid \tau \subseteq F\}$ \citep{Munkres}. Now we discuss shellable nonpure complexes \citep{wachs}.

\begin{definition}
A complex $\Delta$ is shellable if its facets can be arranged in linear order $F_1,F_2,\ldots,F_t$ such that the subcomplex $(\bigcup_{r<s} \bar{F}_{r})\cap \bar{F}_s$ is pure and $(\dim F_s-1)$-dimensional for all $s = 2,\ldots, t$. Such an ordering of facets is called a shelling.
\end{definition}

The following lemma is a reformulation of the previous definition.

\begin{lemma}\label{L2.1}
An order $F_{1}, F_{2}, \ldots, F_{t}$ of the facets of $\Delta$ is a shelling if and only if for every $q$ and $s$ with $1 \leq q<s \leq t$ there is an $r$ with $1 \leq r<s$ and an $x \in F_{s}$ such that $F_{q} \cap F_{s} \subseteq F_{r} \cap F_{s}=F_{s}\setminus\{x\}$.
\end{lemma}

Given a shelling $F_{1}, F_{2}, \ldots, F_{t}$ of $\Delta$, with successively generated subcomplexes $\Delta_{s}=\bigcup_{r=1}^s \bar{F}_{r}$, the restriction of facet $F_{s}$ is
\[
\mathcal{R}(F_{s})=\{x \in F_{s} \mid F_{s}\setminus\{x\} \in \Delta_{s-1}\}.
\]
$\mathcal{R}$ restricts a facet to the vertices that do not obstruct $F_s$ from being in $\Delta_{s-1}$ when deleted. Let
\[
\Gamma_{j,k}=\{F \in \Delta : |F|=j , |\mathcal{R}(F)|=k\}.
\]

We use the following consequence of the homology theory of shellable nonpure complexes developed by \citet{wachs}.

\begin{theorem}\label{T2.2}
Let $\sigma_{F}$ be a $j-1$ cycle such that $\sigma_F$ contains a unique $F \in \Gamma_{j,j}$, then the classes $[\sigma_F]$ are a basis for $\widetilde{H}_{j-1}(\Delta)$ for  $1 \leq j \leq n$ and $\widetilde{H}_{j-1}(\Delta)\cong \mathbb{Z}^{|\Gamma_{j,j}|}$ for  $1 \leq j \leq n.$
\end{theorem}

Now we prove that $D_n$ is nonpure shellable.

\begin{lemma}\label{L3.2}
$D_n$ is a shellable nonpure complex.
\end{lemma}

\begin{proof}
We order the facets of $D_n$ by decreasing number of blocks.  Thus, if
$|F_q|>|F_s|$, then $F_q$ precedes $F_s$; facets with the same number of blocks are ordered arbitrarily.

We verify the criterion in Lemma~\ref{L2.1}.  Let $F_q$ and $F_s$ be distinct facets with $F_q$ preceding $F_s$.  Then $|F_q|\geq |F_s|$.  Remove the blocks common to the two partitions and set
\[
U=[n]\setminus \bigcup (F_q\cap F_s).
\]
The collections
\[
F_q\setminus F_s \qquad \text{and}\qquad F_s\setminus F_q
\]
are both partitions of the same set $U$.  Since the common blocks have been removed from both facets, the inequality $|F_q|\geq |F_s|$ implies
\[
|F_q\setminus F_s|\geq |F_s\setminus F_q|.
\]
We claim that $F_s\setminus F_q$ contains a non-singleton block.  If every block of $F_s\setminus F_q$ were a singleton, then $F_s\setminus F_q$ would be the discrete partition of $U$, so
\[
|F_s\setminus F_q|=|U|.
\]
The preceding inequality would force $|F_q\setminus F_s|\geq |U|$.  But no partition of $U$ has more than $|U|$ blocks, so $F_q\setminus F_s$ would also be the discrete partition of $U$.  Hence
\[
F_q\setminus F_s=F_s\setminus F_q,
\]
and therefore $F_q=F_s$, a contradiction.  Thus, there exists a non-singleton block
\[
B\in F_s\setminus F_q.
\]

Choose a decomposition $B=B'\sqcup B''$ into two nonempty disjoint subsets, and define
\[
F_r=(F_s\setminus\{B\})\cup\{B',B''\}.
\]
Then $F_r$ is a partition of $[n]$ with one more block than $F_s$, so $F_r$ precedes $F_s$ in the chosen order.  Moreover,
\[
F_r\cap F_s=F_s\setminus\{B\}.
\]
Since every block common to $F_q$ and $F_s$ is still present in $F_s\setminus\{B\}$, we have
\[
F_q\cap F_s\subseteq F_r\cap F_s=F_s\setminus\{B\}.
\]
This is exactly the condition in Lemma~\ref{L2.1}.  Therefore, the displayed ordering is a shelling of $D_n$.
\end{proof}

\section{Homology and Cross-Polytope Cycles}

\subsection{Homology of $D_n$}

Earlier in the examples we noted that $\widetilde{H}_1(D_n)$ is isomorphic to $\mathbb{Z}^{D(n,2,2)}$. We now prove the general statement: $\widetilde{H}_{j-1}(D_n)$ is isomorphic to $\mathbb{Z}^{D(n,j,j)}$, for $1\leq j\leq n$.

\begin{lemma}\label{L3.3}
Let $\Delta=D_n$. Then $\Gamma_{j,k}=D_{n,j,k}$, so $|\Gamma_{j,k}|=D(n,j,k)$.
\end{lemma}

\begin{proof}
Recall that $\Delta_{s-1}=\bigcup_{r<s}\overline{F_r}$ is the subcomplex generated by the facets preceding $F_s$ in the shelling order. Let $F_s$ be a facet of $D_n$, and write
\[
F_s=\{B_1,B_2,\ldots,B_j\}.
\]
Thus $F_s$ is a partition of $[n]$ into $j$ blocks. We show that the restriction set $\mathcal R(F_s)$ consists exactly of the non-singleton blocks of $F_s$.

First, suppose that $B\in F_s$ is non-singleton. Choose a partition $\{B',B''\}$ of $B$ into two nonempty disjoint subsets, and set
\[
F_r=(F_s\setminus\{B\})\cup\{B',B''\}.
\]
Then $F_r$ is a partition of $[n]$ with one more block than $F_s$. Hence, $F_r$ appears before $F_s$ in the shelling order. Since
\[
F_s\setminus\{B\}\subseteq F_r,
\]
we have $F_s\setminus\{B\}\in \Delta_{s-1}$. Therefore $B\in\mathcal R(F_s)$.

Now suppose that $B\in F_s$ is a singleton block, say $B=\{i\}$. The blocks of $F_s\setminus\{B\}$ already partition $[n]\setminus\{i\}$. Any facet of $D_n$ containing $F_s\setminus\{B\}$ must be a partition of $[n]$ that contains all of these blocks. The only element not yet covered is $i$, and so the only possible additional block is $\{i\}$. Thus the only facet containing $F_s\setminus\{B\}$ is $F_s$ itself. Hence $F_s\setminus\{B\}\notin \Delta_{s-1}$, and $B\notin\mathcal R(F_s)$.

Therefore
\[
\mathcal R(F_s)=\{B\in F_s: |B|\geq 2\}.
\]
It follows that $|\mathcal R(F_s)|=k$ if and only if $F_s$ has exactly $k$ non-singleton blocks. Thus, the facets in $\Gamma_{j,k}$ are precisely the partitions of $[n]$ into $j$ blocks with exactly $k$ non-singleton blocks. This is exactly the set $D_{n,j,k}$, and hence $|\Gamma_{j,k}|=D(n,j,k)$.
\end{proof}

We are now ready to state the homology of $D_n$.

\begin{theorem}\label{T:homology-main}
For $1\leq j\leq n$,
\[
\widetilde{H}_{j-1}(D_n)\cong \mathbb{Z}^{D(n,j,j)}.
\]
\end{theorem}

\begin{proof}
Theorem~\ref{T2.2} says
\[
\widetilde{H}_{j-1}(D_n)\cong \mathbb{Z}^{|\Gamma_{j,j}|}.
\]
By Lemma~\ref{L3.3}, $\Gamma_{j,j}=D_{n,j,j}$. Therefore
\[
\widetilde{H}_{j-1}(D_n)\cong \mathbb{Z}^{D(n,j,j)}.
\]
\end{proof}

\subsection{Basis of $\widetilde{H}_{j-1}(D_n)$}

Earlier, in the examples section, we saw that the basis elements for
$\widetilde{H}_1(D_n)$ can be represented by cycles that are combinatorially
equivalent to the boundary of a $2$-dimensional cross-polytope. We now make this
construction precise.

Consider the non-singleton partition of size $2$
\[
\{12,34\}.
\]
To construct a cycle containing this partition, we choose one element from each
of its non-singleton blocks. For example, choose $1$ from the block $12$ and
$4$ from the block $34$. Replacing each chosen block by the singleton block formed from its chosen element in
all possible ways gives the formal sum
\[
\{12,34\}
+\{1,34\}
+\{12,4\}
+\{1,4\}.
\]
This formal sum is a $1$-cycle. Moreover, it contains exactly one
non-singleton partition of size $2$, namely $\{12,34\}$.

There are four possible choices of elements associated with the
partition $\{12,34\}$:
\[
1,3, \qquad 1,4, \qquad 2,3, \qquad 2,4.
\]
Each choice gives a corresponding $1$-cycle. Figure~\ref{F3}\subref{F3a} shows
the subcomplex of $D_4$ containing all of these choices. From the figure, one
can see that the resulting cycles are homologous; equivalently, any two of them
differ by a boundary. Later, the theory of shellable nonpure complexes will
give this conclusion in general.

This example suggests a general method for constructing a $1$-cycle containing
exactly one prescribed non-singleton partition of size $2$. More importantly,
it suggests a higher-dimensional construction: from any non-singleton partition
of size $j$, one can construct a $(j-1)$-cycle by choosing one element
from each non-singleton block and forming the corresponding singleton blocks.
These cycles will be the cross-polytope cycles that appear in the homology basis.

Before giving the general construction, we describe the first nontrivial
$2$-dimensional case explicitly. The first value of $n$ for which such a cycle
can occur is $n=6$, because this is the first value for which there exists a
non-singleton partition of size $3$. Consider the partition
\[
\{12,34,56\}.
\]
Choose one element from each block; for instance, choose
\[
1\in \{12\}, \qquad 3\in \{34\}, \qquad 5\in \{56\}.
\]
Now replace any subset of the three non-singleton blocks by the corresponding
singleton block. This produces the formal sum
\[
\begin{aligned}
&\{12,34,56\}
+\{1,34,56\}
+\{12,3,56\}
+\{12,34,5\} \\
&\quad
+\{1,3,56\}
+\{1,34,5\}
+\{12,3,5\}
+\{1,3,5\}.
\end{aligned}
\]
This formal sum is a $2$-cycle. Combinatorially, it is the boundary of a
$3$-dimensional cross-polytope.

As in the $1$-dimensional case, different choices of singleton representatives
produce different cycles. For example, one could choose $2$ instead of $1$ from
the block $12$, or $4$ instead of $3$ from the block $34$, and so on. These
different cycles are homologous: any two of them differ by a boundary. This can
be verified directly in this example by an explicit computation, but the theory
of nonpure shellability will establish the corresponding statement in full
generality.

Let $F = \{B_1, B_2,\ldots, B_j\}\in D_{n,j,j}$.  We now construct a $(j-1)$-cycle from $F$. For each $i=1,\dots,j$, choose an element $b_i\in B_i$, and write
\[
B_i^+:=B_i, \qquad B_i^-:=\{b_i\}.
\]
Note that $B_i^-\neq B_i^+$ because $B_i$ is non-singleton. Define
\[
HO(F)= \{\{B_1^{\epsilon_1}, B_2^{\epsilon_2},\ldots, B_j^{\epsilon_j}\}: (\epsilon_1,\epsilon_2,\ldots,\epsilon_j)\in\{+,-\}^j\}.
\]
That is, $HO(F)$ is the collection of all partial partitions obtained from $F$ by replacing any subset of the blocks $B_i$ with the corresponding singleton $B_i^-$.

\begin{theorem}\label{T3.5}
Let $F=\{B_1,B_2,\ldots,B_j\}\in D_{n,j,j}$, and $\overline{HO}(F)$ be the subcomplex of $D_n$ generated by $HO(F)$.  Then $\overline{HO}(F)$ is combinatorially isomorphic to the boundary complex of the $j$-dimensional cross-polytope.  In particular, the cycle $HO(F)=\sigma_F$ of $\overline{HO}(F)$ determines a homology class in $\widetilde H_{j-1}(D_n)$, and the classes $[\sigma_F]$, as $F$ ranges over $D_{n,j,j}$, form a basis for $\widetilde H_{j-1}(D_n)$.
\end{theorem}

\begin{proof}
It is enough to identify $\overline{HO}(F)$ with the boundary complex of the standard $j$-dimensional cross-polytope and then apply Theorem~\ref{T2.2}. Let $B_i^0$ denote the convention that the $i$th entry is omitted from the partial partition. Then the faces of $\overline{HO}(F)$ are precisely
\[
\{\{B_1^{\epsilon_1},B_2^{\epsilon_2},\ldots,B_j^{\epsilon_j}\}:
(\epsilon_1,\ldots,\epsilon_j)\in\{+,0,-\}^j\},
\]
where entries with exponent $0$ are omitted.

The boundary complex of the standard $j$-dimensional cross-polytope has vertices
\[
\pm e_1,\ldots,\pm e_j.
\]
Its faces are obtained by choosing at most one of $e_i$ and $-e_i$ for each coordinate $i$. Equivalently, its face poset is indexed by sign vectors
\[
(\epsilon_1,\ldots,\epsilon_j)\in\{+,0,-\}^j,
\]
where $0$ means that no vertex from the $i$th antipodal pair is chosen. The order is given by
\[
(\epsilon_1,\ldots,\epsilon_j)\leq (\delta_1,\ldots,\delta_j)
\]
if and only if, for every $i$, either $\epsilon_i=0$ or $\epsilon_i=\delta_i$.

The map
\[
\{B_1^{\epsilon_1},B_2^{\epsilon_2},\ldots,B_j^{\epsilon_j}\}
\longmapsto
(\epsilon_1,\epsilon_2,\ldots,\epsilon_j)
\]
is therefore a poset isomorphism from the face poset of $\overline{HO}(F)$ to the face poset of the boundary complex of the standard $j$-dimensional cross-polytope. 

By construction, the only facet of $D_n$ supported in $\sigma_F$ is $F$ itself: replacing any $B_i$ by $B_i^-$ leaves the elements of $B_i\setminus B_i^-$ uncovered, so the resulting partial partition is not a partition of $[n]$ and hence is not a facet of $D_n$. Thus, $\sigma_F$ contains a unique element of $\Gamma_{j,j}$. Applying Theorem~\ref{T2.2}, the classes $[\sigma_F]$ form a basis for $\widetilde H_{j-1}(D_n)$, with one basis element for each $F\in D_{n,j,j}$.
\end{proof}

\begin{remark}
The homology class $[\sigma_F]$ does not depend on the choice of elements $b_i\in B_i$ used in the construction. If $\sigma_F'$ is the cycle obtained from a different choice, then $\sigma_F-\sigma_F'$ is supported in the subcomplex
\[
D_n^*:=D_n\setminus \bigcup_{j\geq 1}\Gamma_{j,j}.
\]
By \citet[Theorem~4.1]{wachs}, $D_n^*$ is contractible, so $[\sigma_F]=[\sigma_F']$.
\end{remark}

\section{Counting Non-singleton Partitions}

The non-singleton partitions enumerate the homology basis elements. In this section, we deduce a formula for the set of partitions of $[n]$ of size $j$ containing exactly $k$ non-singleton blocks. We denote this set by ${D_{n,j,k}}$ and the size of $D_{n,j,k}$  by ${D(n,j,k)}.$

Using the Stirling numbers of the second kind, denoted by ${S(n,j)}$, and M\"obius inversion on the Boolean lattice \citep{StanleyEC1,Rota}, we deduce an explicit formula for $D(n,j,k)$. \citet{bona} use generating functions to calculate $D(n,j,j)$.

\begin{lemma}
\[
D(n,j,k)=\binom{n}{j-k}\sum_{i \leq n-(j-k)} (-1)^{n-(j-k)-i}\binom{n-(j-k)}{i} S(i,j+i-n).
\]
\end{lemma}

\begin{proof}
Fix $K\subseteq [n]$ and an integer $j\geq 0$. For each size-$j$ partition $\pi$ of $K$, let $T(\pi)\subseteq K$ denote the set of elements lying in singleton blocks of $\pi$. We define two functions on $\mathcal{P}(K)$,
\[
\mathcal{S}_{K,j}(I)\;=\;\bigl|\{\pi:\;\pi\text{ is a size-}j\text{ partition of }K\text{ with }T(\pi)\supseteq K\setminus I\}\bigr|,
\]
and
\[
\mathcal{D}_{K,j}(I)\;=\;\bigl|\{\pi:\;\pi\text{ is a size-}j\text{ partition of }K\text{ with }T(\pi)= K\setminus I\}\bigr|.
\]
That is, $\mathcal{S}_{K,j}(I)$ counts size-$j$ partitions of $K$ in which every element of $K\setminus I$ forms a singleton block, while $\mathcal{D}_{K,j}(I)$ counts those in which the singleton blocks are exactly $\{\{x\}: x\in K\setminus I\}$.

In a partition counted by $\mathcal{S}_{K,j}(I)$, the elements of $K\setminus I$ are forced singletons, so the remaining $|I|$ elements are partitioned into $j-|K\setminus I|$ blocks (with no further restriction). Hence
\[
\mathcal{S}_{K,j}(I)\;=\;S(|I|,\,j-|K\setminus I|).
\]
On the other hand, for each size-$j$ partition $\pi$ of $K$, set $I'=K\setminus T(\pi)$; then $\pi$ is counted by $\mathcal{D}_{K,j}(I')$, and is counted by $\mathcal{S}_{K,j}(I)$ if and only if $T(\pi)\supseteq K\setminus I$, i.e., if and only if $I'\subseteq I$. Therefore
\[
\mathcal{S}_{K,j}(I)\;=\;\sum_{I'\subseteq I}\mathcal{D}_{K,j}(I'),\qquad \text{for all }I\subseteq K.
\]
By M\"obius inversion in the Boolean lattice $\mathcal{P}(K)$, where $\mu(I',I)=(-1)^{|I|-|I'|}$, we obtain
\[
\mathcal{D}_{K,j}(I)\;=\;\sum_{I'\subseteq I}(-1)^{|I|-|I'|}\,\mathcal{S}_{K,j}(I').
\]
Setting $I=K$, $|K|=k$, and grouping subsets of $K$ by their size $|I'|=i$, of which there are $\binom{k}{i}$, gives
\[
\mathcal{D}_{K,j}(K)\;=\;\sum_{i\leq k}(-1)^{k-i}\binom{k}{i}\,S(i,\,j-k+i).
\]
Note that $\mathcal{D}_{K,j}(K)$ counts size-$j$ partitions of $K$ with $T(\pi)=\emptyset$, i.e., with no singleton blocks. Hence $\mathcal{D}_{K,j}(K)=D(k,j,j)$, and we obtain
\[
D(k,j,j)\;=\;\sum_{i\leq k}(-1)^{k-i}\binom{k}{i}\,S(i,\,j-k+i).
\]

For the general formula, observe that there are exactly $\binom{n}{j-k}$ ways to choose which $j-k$ elements of $[n]$ form the singleton blocks of a partition counted by $D(n,j,k)$. For each such choice, the remaining $n-(j-k)$ elements must be partitioned into $k$ non-singleton blocks, contributing $D(n-(j-k),k,k)$ partitions. Therefore
\[
D(n,j,k)\;=\;\binom{n}{j-k}\,D(n-(j-k),\,k,\,k).
\]
Substituting the formula for $D(n-(j-k),k,k)$ obtained above (with $k$ replaced by $n-(j-k)$ and $j$ replaced by $k$, noting that $k-(n-(j-k))+i=j+i-n$) gives
\[
D(n,j,k)\;=\;\binom{n}{j-k}\sum_{i\leq n-(j-k)}(-1)^{n-(j-k)-i}\binom{n-(j-k)}{i}S(i,\,j+i-n).\qedhere
\]
\end{proof}

\section{Automorphisms and Further Geometric Questions}

We conclude with two structural observations.  First, the full automorphism group of $D_n$ is induced by the natural action of the symmetric group on $[n]$.  Second, the complex raises natural geometric realization questions.

\subsection{Automorphisms of $D_n$}

\begin{theorem}\label{T3.7}
For $n\geq 1$, the automorphism group of the simplicial complex $D_n$ is isomorphic to $S_n$.
\end{theorem}

\begin{proof}
Each permutation $\pi\in S_n$ acts on the vertex set of $D_n$ by
\[
B\mapsto \pi(B),
\]
where $B\subseteq [n]$ is nonempty.  Since disjointness is preserved under permutations, this action sends simplices to simplices.  Thus each $\pi$ induces a simplicial automorphism of $D_n$, and we obtain an injective homomorphism
\[
S_n\hookrightarrow \operatorname{Aut}(D_n).
\]

To prove surjectivity, let $\varphi\in \operatorname{Aut}(D_n)$.  The vertices $\{1\},\ldots,\{n\}$ form the unique simplex of dimension $n-1$, since a simplex in $D_n$ is a collection of pairwise disjoint nonempty subsets of $[n]$, and the only such collection of size $n$ consists of all singleton subsets.  Therefore $\varphi$ permutes the set $\{\{1\},\ldots,\{n\}\}$.  Hence there exists $\pi\in S_n$ such that
\[
\varphi(\{i\})=\{\pi(i)\}
\]
for all $i\in[n]$.

Now let $B\subseteq[n]$ be nonempty.  For each $i\in[n]$, the pair $\{B,\{i\}\}$ is an edge of $D_n$ if and only if $i\notin B$.  Since $\varphi$ preserves adjacency,
\[
i\notin B
\quad\Longleftrightarrow\quad
\{\pi(i)\}\text{ is disjoint from }\varphi(B).
\]
Equivalently,
\[
i\in B
\quad\Longleftrightarrow\quad
\pi(i)\in \varphi(B).
\]
Thus $\varphi(B)=\pi(B)$.  Therefore every automorphism of $D_n$ is induced by a permutation of $[n]$, and so
\[
\operatorname{Aut}(D_n)\cong S_n.
\]
\end{proof}

\subsection{Geometric realization questions}

Recall that $D_n$ has dimension $n-1$.  Thus $\mathbb R^{n-1}$ is the smallest Euclidean space in which one could hope to realize $D_n$ as a geometric simplicial complex without increasing the ambient dimension beyond the dimension of the complex itself.  By a \emph{linear embedding} of $D_n$ into $\mathbb R^k$ we mean a continuous injective map of the geometric realization $|D_n|\to \mathbb R^k$ that is affine on each closed simplex. This leads to the following problem.

\begin{problem}\label{P4.1}
Determine whether $D_n$ admits a linear embedding
\[
\varphi:D_n\hookrightarrow \mathbb R^{n-1}
\]
that realizes $D_n$ as a geometric simplicial complex.
\end{problem}

Since $\operatorname{Aut}(D_n)\cong S_n$, one can also ask for realizations that respect the full combinatorial symmetry of the complex.  Following the standard viewpoint of group actions on geometric objects \citep{Bredon}, we formulate the equivariant version as follows.

\begin{definition}\label{D4.2}
A linear embedding $\varphi:D_n\to \mathbb R^k$ is called \emph{$S_n$-equivariant} if there exists a homomorphism
\[
\rho:S_n\to \operatorname{Isom}(\mathbb R^k)
\]
such that, for all $\pi\in S_n$ and all simplices $\sigma\in D_n$,
\[
\varphi(\pi\cdot\sigma)=\rho(\pi)\varphi(\sigma).
\]
\end{definition}

\begin{problem}\label{P4.3}
Determine the smallest $k$ for which there exists an $S_n$-equivariant linear embedding
\[
\varphi:D_n\hookrightarrow \mathbb R^k.
\]
\end{problem}

A related question asks whether the homology basis constructed above can be realized geometrically in a compatible way.

\begin{problem}\label{P4.4}
Does there exist a linear embedding of $D_n$ into Euclidean space such that the cycles constructed in Theorem~\ref{T3.5} are realized as geometric cross-polytopes \citep{Ziegler}?  If so, what is the minimal dimension in which this can be achieved?
\end{problem}

These questions indicate that the containment-order partial partition complex has additional geometric structure beyond the homological description proved here.

%

%

\end{document}